\documentclass[11pt]{article}
\usepackage{graphicx, colordvi}
\usepackage{dsfont}
\usepackage{epsfig}
\usepackage{mathrsfs}
\usepackage{amssymb,amsfonts,amsmath,amsthm,cite,color}
\usepackage{dsfont}
\usepackage{epsfig}
\usepackage{mathrsfs}
\usepackage{longtable}
\usepackage{enumerate}
\usepackage{geometry}
\usepackage{hyperref}
\hypersetup{
colorlinks=true,
linkcolor=cyan,
filecolor=blue,
urlcolor=red,
citecolor=green,
}

\newtheorem{theo}{Theorem}

\newtheorem{coro}[theo]{Corollary}

\newtheorem{lem}[theo]{Lemma}

\makeatletter \@addtoreset{equation}{section}
\@addtoreset{theo}{section} \makeatother

\newcommand{\bN} { {\mathbb{N}}}

\newcommand{\bZ} { {\mathbb{Z}}}

\newcommand{\bK} { {\mathbb{K}}}

\DeclareMathOperator{\ann}{ann}

\newcommand{\la} { {\langle}}
\newcommand{\ra} { {\rangle}}
\newcommand{\re}{\noindent{\bfseries Remark. }}

\def\qed{\hfill \rule{4pt}{7pt}}
\def\pf{\noindent {\it Proof.} }

\begin{document}
\begin{center}
 {\large \bfseries Congruences for sums of Delannoy numbers and polynomials}
\end{center}

\begin{center}
{  Rong-Hua Wang}$^{1}$ and {Michael X.X. Zhong}$^{2}$

   $^1$School of Mathematical Sciences\\
   Tiangong University \\
   Tianjin 300387, P.R. China\\
   wangronghua@tiangong.edu.cn \\[10pt]

   $^2$School of Science\\
   Tianjin University of Technology \\
   Tianjin 300384, P.R. China\\
   zhong.m@tjut.edu.cn
\end{center}

\vskip 6mm \noindent {\bf Abstract.}
In this paper, we apply the power-partible reduction to study arithmetic properties of sums involving Delannoy numbers $D_k$ and polynomials $D_k(z)$.
Let $v\in\bN$ and $p$ be an odd prime.
It is proved that, for any $z\in\bZ\setminus\{0,-1\}$, there exist $c_v\in z^{-v}\bZ[z]$ and $\tilde{c}_v\in (z+1)^{-v}\bZ[z]$, both free of $p$ and can be determined mechanically, such that
\begin{equation*}
\sum_{k=0}^{p-1}(2k+1)^{2v}D_k(z)\equiv
c_v \left(\frac{-z}{p}\right) \pmod {p}
\end{equation*}
if $\gcd(p,z)=1$ and
\begin{equation*}
\sum_{k=0}^{p-1}(-1)^k(2k+1)^{2v}D_k(z)\equiv
\tilde{c}_v \left(\frac{z+1}{p}\right) \pmod {p}
\end{equation*}
if $\gcd(p,z+1)=1$.
Here $(-)$ denotes the Legendre symbol.

When $n$ is a power of $2$, we find there exist odd integers $\rho_v$ and even integers $\tilde{\rho}_v$, both independent of $n$ and can be determined mechanically, such that
\[
\sum_{k=0}^{n-1}(2k+1)^{2v+1}D_k\equiv \rho_v n \pmod {n^3}
\]
and
\[
\sum_{k=0}^{n-1}(-1)^k(2k+1)^{2v+1}D_k\equiv \tilde{\rho}_v n^2 \pmod {n^3}.
\]
The case $v=1$ in the last congruence confirms a conjecture of Guo and Zeng in 2012.


\section{Introduction}

Let $\bN$ and $\bZ^{+}$ denote the set of nonnegative integers and positive integers respectively.
The central Delannoy numbers $D_n$, which appear ubiquitously in enumerative combinatorics, are defined as
\begin{equation*}
D_{n}=\sum\limits_{k=0}^{n}\binom{n}{k}\binom{n+k}{k}.
\end{equation*}
It is well-known that $D_n$ is the number of lattice paths from the point $(0, 0)$ to $(n, n)$ with steps choosing from the set $\{(1, 0), (0, 1), (1, 1)\}$.
One can consult item A001850 in the OEIS
\cite{OEIS}
for other classical combinatorial interpretations.

Sun \cite{Sun2011,Sun2014} discovered many interesting arithmetic properties of $D_n$.
In particular, he showed that, for any positive integer $n$,
\[
\frac{1}{n^2}\sum_{k=0}^{n-1}(2k+1)D_k^2\in\bZ
\]
and that, for any prime $p>3$,
\[
\sum_{k=1}^{p-1}\frac{D_k}{k^2}\equiv (-1)^{(p-1)/2}2E_{p-3}
\pmod p
\]
where $E_k$ are the Euler numbers.

Sun \cite{Sun2011} introduced the Delannoy polynomials
\[
D_n(z)=\sum_{k=0}^{n}\binom{n}{k}\binom{n+k}{k}z^k
,\quad n\in\bN,
\]
and proved that \cite{Sun2014} for any $n\in\bZ^{+}$,
\begin{equation*}\label{eq:(2k+1)Dk}
\frac{1}{n}\sum_{k=0}^{n-1}(2k+1)D_k(z)
=\sum_{k=0}^{n-1}\binom{n}{k+1}\binom{n+k}{k}z^k
\in\bZ[z],
\end{equation*}
and thus
\[
\sum_{k=0}^{n-1}(2k+1)D_k(z)\equiv 0\pmod n.
\]


%
%

The Ap\'ery numbers \cite{Apery1979,Poorten1979} given by
\[
A_n=\sum_{k=0}^{n}\binom{n}{k}^2\binom{n+k}{k}^2,\quad n\in\bN
\]
arose in Ap\'ery's proof of irrationality of $\zeta(3)=\sum_{n=1}^{\infty}\frac{1}{n^3}$.
Congruences for Ap\'ery numbers have been investigated extensively.
For example, Gessel \cite{Gessel} proved that
\[
A_{pn}\equiv A_n \pmod{n^3}
\]
for any $n\in\bZ^{+}$ and prime $p\geq 5$.
Xia and Sun \cite{XiaSun2022} confirmed the conjecture posed by Sun \cite{Sun2016} that
for each $v\in\bN$ and prime $p>3$, there is a $p$-adic integer $c_v$ such that
\[
\sum_{k=0}^{p-1}(2k+1)^{2v+1}(-1)^kA_k\equiv c_v p \left(\frac{p}{3}\right) \pmod {p^3},
\]
where $(-)$ denotes the Legendre symbol.
The authors \cite{WZ2024} derived that for each $v\in\bN$ and prime $p>3$, there exists an integer $\tilde{c}_v$ such that
\[
\sum_{k=0}^{p-1}(2k+1)^{2v+1}A_k\equiv \tilde{c}_v p\pmod {p^3}.
\]

In pursuing a more general study of arithmetic properties of Ap\'ery numbers,
Sun \cite[(1.1)]{Sun2012} introduced the Ap\'ery polynomials
\[
A_n(z)=\sum_{k=0}^{n}\binom{n}{k}^2\binom{n+k}{k}^2z^k, \quad n\in\bN.
\]
He proved that for any $n\in\bZ^{+}$,
\[
\sum_{k=0}^{n-1}(2k+1)A_k(z)\equiv 0\pmod n
\]
and conjectured that for any
$\varepsilon\in\{1,-1\},m,n\in\bZ^{+}$ and $z\in\bZ$,
\[
\sum_{k=0}^{n-1}(2k+1)\varepsilon^kA_k(z)^m\equiv 0\pmod n.
\]
Guo and Zeng \cite{GuoZeng2012} confirmed Sun's conjecture for the case $\varepsilon=-1,m=1$
and further introduced the Schmidt polynomial
\[
S_n^{(r)}(z)=\sum_{k=0}^{n}\binom{n}{k}^{r}\binom{n+k}{k}^{r}z^k,\quad n\in\bN,
\]
which is a natural generalization of the Delannoy polynomial and Ap\'ery polynomial.
Then they showed that, for any $n\in \bZ^{+},v\in\bN,r\ge 2$ and $z\in\bZ$,
\begin{equation*}\label{eqn:Guo--Zeng}
\sum_{k=0}^{n-1}\varepsilon^k(2k+1)^{2v+1}S^{(r)}_k(z)\equiv 0\pmod n,
\end{equation*}
and conjectured that, for any $m,n\in \bZ^{+},r\ge 2$ and $z\in\bZ$,
\begin{equation*}\label{eqn:Guo--Zeng conjecture}
\sum_{k=0}^{n-1}\varepsilon^k(2k+1)S^{(r)}_k(z)^m\equiv 0\pmod n.
\end{equation*}
This conjecture was confirmed by Pan \cite{Pan2014} with \emph{q}-congruences.
In 2018, Chen and Guo \cite{ChenGuo2018} defined the multi-variate Schmidt polynomial as
\[
S_k^{(r)}(x_0,x_1,\ldots,x_k)=\sum_{i=0}^{k}\binom{k+i}{2i}^r \binom{2i}{i}x_i,
\]
and proved that
\begin{equation*}\label{eq:multi-sch}
\sum_{k=0}^{n-1}\varepsilon^k(2k+1)^{2v+1}S_k^{(r)}(x_0,x_1,\ldots,x_k)^m\equiv 0 \pmod n.
\end{equation*}
All the mentioned congruences contain $(2k+1)^s$ with odd $s$.
The first part of this paper considers the arithmetic properties for sums of $(2k+1)^{s}D_k(z)$ with even $s$.
The following theorem illustrates that one can reduce the problem to the case of $s=0$.

%

\begin{theo}\label{th:Delanoy even}
Let $n\in\bZ^+$. For each $v\in\bN$ and $z\in\bZ\setminus\{0,-1\}$, we have
\begin{equation}\label{eq:(2k+1)^{2r+2}}
\sum_{k=0}^{n-1}(2k+1)^{2v}D_k(z)\equiv
c_v \sum_{k=0}^{n-1}D_k(z) \pmod {n}
\end{equation}
when $\gcd(n,2z)=1$ and
\begin{equation}\label{eq:-(2k+1)^{2r+2}}
\sum_{k=0}^{n-1}(-1)^k(2k+1)^{2v}D_k(z)\equiv
\tilde{c}_v \sum_{k=0}^{n-1}(-1)^kD_k(z) \pmod {n}
\end{equation}
when $\gcd(n,2(z+1))=1$.
Here $c_0=\tilde{c}_0=1$ and when $v\geq 1$,
\begin{equation}\label{eq:c_v}
c_v=\frac{1}{z}\sum_{j=1}^{v}e_j^{(2v-1)}c_{v-j}\in z^{-v}\bZ[z]
\text{ and }
\tilde{c}_v=-\frac{1}{z+1}\sum_{j=1}^{v}e_j^{(2v-1)}\tilde{c}_{v-j}\in (z+1)^{-v}\bZ[z],
\end{equation}
where
\begin{equation}\label{eq:ejs}
e^{(s)}_j=\binom{s}{2j}2^{2j-1}+\binom{s}{2j-1}2^{2j-2} \in\bZ
\end{equation}
for $s\in\bN$ and $j=1,2,\ldots,\left\lfloor \dfrac{s+1}{2}\right\rfloor$.
\end{theo}

In 2014, Sun \cite[Theorem 1.2(i)]{Sun2014} proved that when $p$ is an odd  prime and $b,c\in\bZ$, for any integer $m$ coprime with $p$, we have
\begin{equation}\label{eq:Tk(b,c)}
\sum_{k=0}^{p-1}\frac{T_k(b,c)}{m^k}
\equiv\left(\frac{(m-b)^2-4c}{p}\right) \pmod p.
\end{equation}
Here $(-)$ denotes the Legendre symbol and $T_n(b,c)$ is the generalized central trinomial coefficients defined by
\[
T_n(b,c)=\sum_{k=0}^{\lfloor n/2\rfloor}
\binom{n}{2k}\binom{2k}{k}b^{n-2k}c^k.
\]
Sun \cite{Sun2014} showed that $T_k(b,c)=D_k(z)$ if $b=2z+1$ and $c=z^2+z$ for any $z\in\bZ$.
Then taking $m=1$ and $m=-1$ respectively in \eqref{eq:Tk(b,c)} gives
\begin{equation}\label{eq:(2k+1)^0}
\sum_{k=0}^{p-1} D_k(z)\equiv \left(\frac{-z}{p}\right) \pmod p
\end{equation}
and
\begin{equation}\label{eq:-(2k+1)^0}
\sum_{k=0}^{p-1} (-1)^k D_k(z)\equiv \left(\frac{z+1}{p}\right) \pmod p.
\end{equation}
Combining Theorem \ref{th:Delanoy even} and congruences \eqref{eq:(2k+1)^0} and \eqref{eq:-(2k+1)^0} together leads to the following.

\begin{theo}\label{th:Maineven}
Let $p$ be an odd prime.
For each $v\in\bN$ and $z\in\bZ\setminus\{0,-1\}$, we obtain that
\begin{equation*}
\sum_{k=0}^{p-1}(2k+1)^{2v}D_k(z)\equiv
c_v \left(\frac{-z}{p}\right) \pmod {p}
\end{equation*}
when $\gcd(p,z)=1$, and that
\begin{equation*}
\sum_{k=0}^{p-1}(-1)^k(2k+1)^{2v}D_k(z)\equiv
\tilde{c}_v \left(\frac{z+1}{p}\right) \pmod {p}
\end{equation*}
when $\gcd(p,z+1)=1$.
Here $c_v$ and $\tilde{c}_v$ are given as in \eqref{eq:c_v}.
\end{theo}

\re
The cases for $z\in\{0,-1\}$ are trivial since $D_k(0)=1$ and $D_k(-1)=(-1)^k$.
%
%

The second part of the paper concentrates on supercongruences of sums involving the Delannoy numbers $D_k$.

\begin{theo}\label{th:1n^3}
Let $v\in\bN$ and $n=2^a$ for some $a\in\bZ^+$.
We have
\begin{equation}\label{eq:1main2}
\sum_{k=0}^{n-1}(2k+1)^{2v+1}D_k\equiv \rho_v n \pmod {n^3}
\end{equation}
and
\begin{equation}\label{eq:-1main2}
\sum_{k=0}^{n-1}(-1)^k(2k+1)^{2v+1}D_k\equiv \tilde{\rho}_v n^2 \pmod {n^3}.
\end{equation}
Here $\rho_v$ is an odd integer given by
\begin{equation}\label{eq:pv}
\rho_v=-2\sum_{j=1}^{v}e_j^{(2v)} y_{v-j}(-1)+1,
\end{equation}
where
$y_0(k)=-\frac{1}{2}$ and
\begin{equation}\label{eq:yvk0}
y_v(k)=-\frac{1}{2}(2k+3)^{2v}+\sum_{j=1}^{v}e_j^{(2v)}y_{v-j}(k),\quad v\geq 1.
\end{equation}
While $\tilde{\rho}_v$ is an even integer given by
\begin{equation}\label{eq:pv2}
\tilde{\rho}_v=2v-\sum_{j=1}^{v}e^{(2v)}_j
\tilde{y}_{v-j}'(-1)\in\bZ,
\end{equation}
where $\tilde{y}_0(k)=\frac{1}{4}$ and when $v\geq 1$
\begin{equation}\label{eq:yvk}
\tilde{y}_v(k)=\frac{(2k+3)^{2v}}{4}-\sum_{j=1}^{v}\frac{e^{(2v)}_j}{2}\tilde{y}_{v-j}(k).
\end{equation}
Here $e^{(2v)}_j$ are even integers given by \eqref{eq:ejs}.
\end{theo}

In particular, taking $v=1$ in \eqref{eq:-1main2} confirms Guo and Zeng's \cite{GuoZeng2012a} conjecture: when $n$ is a power of $2$, one has
\begin{equation}\label{eq:GZconj}
\sum_{k=0}^{n-1}(-1)^k(2k+1)^3D_k\equiv 2n^2 \pmod {n^3}.
\end{equation}

The rest of the paper is organized as follows.
In Section 2, we recall the power-partible reduction and apply it to the proof of Theorem \ref{th:Delanoy even}.
The proof of Theorem \ref{th:1n^3} and arithmetic properties of $D_{2^a\pm1}$ needed in the proof will be given in Section 3.

\section{Power-partible reduction and congruences on Delannoy polynomials}
In this section, we first recall some basic notations and facts concerning power-partible reduction, which was introduced by Hou, Mu and Zeilberger \cite{HouMuZeil2021} for hypergeometric terms and extended to holonomic sequences by the authors in \cite{WZ2022,WZ2024}.
Then the proof of Theorem \ref{th:Delanoy even} along with other arithmetic properties for $D_k(z)$ will be presented.

\subsection{Power-partible reduction}

Let $\bK$ be a field of characteristic $0$ and $\bK[k]$ the polynomial ring over $\bK$.
The set of \emph{annihilators} of a sequence $F(k)$ is denoted by
\begin{equation*}\label{eq:Annihilator}
\ann F(k):=\left\{L=\sum_{i=0}^{J}a_i(k)\sigma^i\in \bK[k][\sigma]\mid L(F(k))=0\right\},
\end{equation*}
where $\sigma$ is the shift operator (that is, $\sigma F(k)=F(k+1)$) and $J\in\bN=\{0,1,2,\ldots\}$ is called the order of the operator $L$ if $a_J(k)\neq 0$.
A sequence $F(k)$ is called \emph{holonomic} if $\ann F(k)\neq \{0\}$.
Given an operator $L=\sum\limits_{i=0}^{J}a_i(k)\sigma^i\in \bK[k][\sigma]$ of order $J$,
the \emph{adjoint} of $L$ is defined by
$
L^{\ast}=\sum_{i=0}^{J}\sigma^{-i}a_i(k),
$
that is,
\[
L^{\ast}(x(k))=\sum_{i=0}^{J}a_i(k-i)x(k-i)
\]
for any $x(k)\in\bK[k]$.
Suppose $L=\sum_{i=0}^{J}a_i(k)\sigma^i\in\ann F(k)$ and $x(k)\in\bK[k]$.
By \cite{ChenHouJin2012} or \cite{Hoeven2018}, one has
\begin{equation}\label{eq: summable}	L^{\ast}(x(k))F(k)=\Delta\left(-\sum_{i=0}^{J-1}u_i(k)F(k+i)\right),
\end{equation}
where
$
u_i(k)=\sum_{j=1}^{J-i}a_{i+j}(k-j)x(k-j),\ i=0,1,2,\ldots,J-1.
$
Summing over $k$ from $0$ to $n-1$ on both sides of identity \eqref{eq: summable} leads to
\begin{equation}\label{eq:rec ann}
\sum_{k=0}^{n-1}L^{\ast}(x(k)) F(k)
=\left(\sum_{i=0}^{J-1}u_i(0)F(i)\right)
-\left(\sum_{i=0}^{J-1}u_i(n)F(n+i)\right).
\end{equation}

The set
\begin{equation*}\label{eq:SL}
S_{L}=\{L^{\ast}(x(k))\mid x(k)\in\bK[k]\}
\end{equation*}
is called the \emph{difference space} corresponding to $L$.
We will take $[p(k)]_L=p(k)+S_L$ as the coset of a polynomial $p(k)$.
The \emph{degree} of $L=\sum\limits_{i=0}^{J}a_i(k)\sigma^i\in \bK[k][\sigma]$ with $a_J(k)\neq 0$ is defined by
\begin{equation*}\label{eq:d and bk}
\deg L=\max_{0\leq \ell \leq J} \{\deg b_\ell(k)-\ell\},
\end{equation*}
where $b_{\ell}(k)=\sum_{j={\ell}}^{J}\binom{j}{\ell}a_{J-j}(k+j-J)$.
Let $d=\deg L$ and
\begin{equation*}\label{eq:nonnegative roots}
R_{L}=\left\{s\in\bN \mid \sum_{\ell=0}^{J}[k^{d+\ell}](b_\ell(k))s^{\underline{\ell}}=0\right\}.
\end{equation*}
Here $[k^{d+\ell}](b_\ell(k))$ denotes the coefficient of $k^{d+\ell}$ in $b_\ell(k)$ and $s^{\underline{\ell}}$ denotes the falling factorial $s(s-1)\cdots(s-\ell+1)$.
Then $L$ is called \emph{nondegenerated} if $R_{L}=\emptyset$, and \emph{degenerated} otherwise.
%

\begin{theo}[Theorem 2.4 in \cite{WZ2024}]\label{thm:power-partible}
Let $L=\sum_{i=0}^{J}a_i(k)\sigma^i\in \bK[k][\sigma]$ with $a_J(k)\neq 0$.
Suppose $L$ is nondegenerated and there exists a $\gamma\in\bK$ such that
\begin{equation}\label{eq:power-reduce-condition}
a_i(\gamma+k)=(-1)^{\deg L} a_{J-i}(\gamma-k-J),\quad i=0,1,\ldots,\lfloor\frac{J}{2}\rfloor.
\end{equation}
Then for any positive integer $m$, we have
\begin{equation*}
  [(k-\gamma)^m]_L\in\la [(k-\gamma)^i]_L\mid i\equiv m \pmod 2, 0\leq i<\deg L\ra.
\end{equation*}
\end{theo}
When conditions in Theorem \ref{thm:power-partible} are satisfied, we say $L$ is \emph{power-partible} with respect to $\gamma$.
If $L\in\ann F(k)$ for some holonomic sequence $F(k)$, one may also say $F(k)$ is \emph{power-partible} with respect to $\gamma$.

%

\subsection{Arithmetic properties of Delannoy polynomials}

In this subsection, we utilize the power-partible reduction to prove Theorem \ref{th:Delanoy even}. We will assume $z\in\bZ\setminus\{0,-1\}$ in the following text.

Let $\varepsilon\in\{1,-1\}$ and $F_k(z)=\varepsilon^k D_k(z)$.
By Zeilberger's algorithm \cite{Zeilberger1991}, we find that
\begin{equation}\label{eq:L Delannoy}
L=(k+2)\sigma^2-\varepsilon(2k+3)(2z+1)\sigma+(k+1)\in \ann F_k(z).
\end{equation}
It is straightforward to check that $L$ is nondegenerated, $d=\deg(L)=1$ and condition \eqref{eq:power-reduce-condition} holds for $\gamma=-\frac{1}{2}$.
Namely, $F_k(z)$ is power-partible with respect to $\gamma$.
Thus for any $v\in\bN$, it follows that
\begin{equation*}\label{eq:Delannoy (2k+1)}
[(2k+1)^{2v}]_L\in\la[(2k+1)^0]_L\ra \quad \text{and} \quad [(2k+1)^{2v+1}]_L\in\la[0]_L\ra.
\end{equation*}

\begin{theo}\label{th:Apery cong}
Let $L$ be given as in \eqref{eq:L Delannoy} and $n\in\bZ^{+}$.
For any polynomial $x(k)$, we have
\begin{equation}\label{eq: D1}
\sum_{k=0}^{n-1} L^{\ast}(x(k))F_k(z)
=n(x(n-1)F_{n-1}(z)-x(n-2)F_n(z)).
\end{equation}
\end{theo}
\begin{proof}
By Eq.~\eqref{eq:rec ann} and the easily checked fact $u_0(0)F_0(z)+u_1(0)F_1(z)=0$, we have
\begin{equation}\label{eq:Apery Delta1}
\sum_{k=0}^{n-1}L^{\ast}(x(k))F_k(z)
=
-\left(u_0(n)F_n(z)+u_1(n)F_{n+1}(z)\right),
\end{equation}
where
$u_0(n)=nx(n-2)-\varepsilon(2 n+1) (2z+1) x(n-1)$ and
$u_1(n)=(n+1) x(n-1)$.
As $L\in\ann D_k(z)$, it is clear that for any $n\geq 1$,
\begin{equation}\label{eq:shift Apery1}
(n+1)F_{n+1}(z)=\varepsilon(2n+1)(2z+1)F_n(z)-nF_{n-1}(z).
\end{equation}
Substituting equality~\eqref{eq:shift Apery1} into \eqref{eq:Apery Delta1} derives ~\eqref{eq: D1}.
\end{proof}
\begin{coro}\label{cor:divisibility-L*}
Let $L$ be given as in \eqref{eq:L Delannoy} and $n$ a positive integer. Then for any polynomial $x(k)$ with integer coefficients, we have
\begin{equation*}\label{eq: cong D1}
\sum_{k=0}^{n-1} L^{\ast}(x(k))F_k(z)
\equiv 0\pmod n.
\end{equation*}
\end{coro}
Recall that $L$ in \eqref{eq:L Delannoy} is power-partible with respect to $\gamma=-\dfrac{1}{2}$.
Let
\begin{equation}\label{eq:xs1}
x_s(k)=(2k+3)^s,\quad s\in\bN.
\end{equation}
We will show that $L^\ast(x_s(k))$ is a linear combination of $(2k+1)^{i}$ with $i\equiv s+1{\pmod 2}$ and that all the coefficients in the combination are even integers.
\begin{lem}\label{lm:Structure}
Let $L$ be given by \eqref{eq:L Delannoy} and $x_s(k)$ given by Eq.~\eqref{eq:xs1}.
Then
\begin{equation}\label{eq:Explicit form of L*(x_s(k))2}
L^\ast(x_s(k))=(-2\varepsilon z-\varepsilon+1)(2k+1)^{s+1}+2\sum_{j=1}^{\left\lfloor \frac{s+1}{2}\right\rfloor}e^{(s)}_j(2k+1)^{s+1-2j},
\end{equation}
where $e^{(s)}_j$ are given as in \eqref{eq:ejs}.
\end{lem}
\pf
For simplicity, let $\ell=2k+1$.
By the definition of $L^\ast$, we have
\begin{align*}
L^\ast(x_s(k))
=&\frac{1}{2}(\ell-1)(\ell-2)^s-\varepsilon(2z+1)\ell^{s+1}
+\frac{1}{2}(\ell+1)(\ell+2)^s\\
=&\frac{1}{2}(\ell-1)\sum_{j=0}^{s}\binom{s}{j}(-2)^{j}\ell^{s-j}+
\frac{1}{2}(\ell+1)\sum_{j=0}^{s}\binom{s}{j}2^{j}\ell^{s-j}-\varepsilon(2z+1)\ell^{s+1}\\
=&\sum_{\substack{j=0\\ j\text{ even}}}^{s}\binom{s}{j}2^{j}\ell^{s-j+1}
+\sum_{\substack{j=0\\ j\text{ odd}}}^{s}\binom{s}{j}2^{j}\ell^{s-j}
-\varepsilon(2z+1)\ell^{s+1}\\
=&(-2\varepsilon z-\varepsilon+1)\ell^{s+1}+\sum_{\substack{j=2\\ j\text{ even}}}^{s}\binom{s}{j}2^{j}\ell^{s-j+1}+\sum_{\substack{j=1\\ j\text{ odd}}}^{s}\binom{s}{j}2^{j}\ell^{s-j}\\
=&(-2\varepsilon z-\varepsilon+1)(2k+1)^{s+1}+2\sum_{j=1}^{\left\lfloor \frac{s+1}{2}\right\rfloor}e^{(s)}_j(2k+1)^{s+1-2j}.\quad \quad\qed
\end{align*}

%

Now we are ready to provide the proof of Theorem \ref{th:Delanoy even}.

\noindent\emph{Proof of \eqref{eq:(2k+1)^{2r+2}} in Theorem \ref{th:Delanoy even}.}
We proceed by induction on $v$.
The case when $v=0$ is trivial with $c_0=1$.
Next, we assume that $v>0$ and \eqref{eq:(2k+1)^{2r+2}} holds for all smaller $v$'s with $c_j\in z^{-j}\bZ[z]$ ($j=0,1,\ldots,v-1$).

By taking $\varepsilon=1$ in equality \eqref{eq:Explicit form of L*(x_s(k))2}, it is easily checked that
\begin{equation}\label{eq:2v}
(2k+1)^{2v}=-\frac{1}{2z} L^{\ast}(x_{2v-1}(k))+\frac{1}{z}\sum_{j=1}^{v}e^{(2v-1)}_j(2k+1)^{2v-2j},
\end{equation}
where $e^{(2v-1)}_j\in\bZ$.
Multiplying by $D_k(z)$ on both sides of \eqref{eq:2v} and summing over $k$ from $0$ to $n-1$ leads to
\[
\sum_{k=0}^{n-1}(2k+1)^{2v}D_k(z)
=-\frac{1}{2z}\sum_{k=0}^{n-1} L^{\ast}(x_{2v-1}(k))D_k(z)
 +\frac{1}{z}\sum_{j=1}^{v}e^{(2v-1)}_j\sum_{k=0}^{n-1}(2k+1)^{2v-2j}D_k(z).
\]
By the assumption, we know
\begin{equation}\label{eq:re1}
\sum_{j=1}^{v}e^{(2v-1)}_j\sum_{k=0}^{n-1}(2k+1)^{2v-2j}D_k(z)
\equiv
\sum_{j=1}^{v}e^{(2v-1)}_j c_{v-j}\sum_{k=0}^{n-1}D_k(z) \pmod n.
\end{equation}
Note also that, by Corollary \ref{cor:divisibility-L*}, when $\gcd(n,2z)=1$,
we have
\begin{equation}\label{eq:re2}
-\frac{1}{2z}\sum_{k=0}^{n-1}L^{\ast}(x_{2v-1}(k))D_k(z) \equiv 0 \pmod n.
\end{equation}
Combining \eqref{eq:re1} and \eqref{eq:re2} together arrives at
\[
\sum_{k=0}^{n-1}(2k+1)^{2v}D_k(z)\equiv
c_v \sum_{k=0}^{n-1}D_k(z) \pmod {n},
\]
where $c_v=\frac{1}{z}\sum_{j=1}^{v}e^{(2v-1)}_j c_{v-j}\in z^{-v}\bZ[z]$.
\qed

\noindent\emph{Proof of \eqref{eq:-(2k+1)^{2r+2}} in Theorem \ref{th:Delanoy even}.}
Taking $\varepsilon=-1$ in equality \eqref{eq:Explicit form of L*(x_s(k))2} leads to
\begin{equation*}\label{eq:2v+2}
(2k+1)^{2v}=\frac{1}{2(z+1)} L^{\ast}(x_{2v-1}(k))-\frac{1}{z+1}\sum_{j=1}^{v}e^{(2v-1)}_j(2k+1)^{2v-2j},
\end{equation*}
where $e^{(2v-1)}_j\in\bZ$.
The rest of the proof is then similar to that of \eqref{eq:(2k+1)^{2r+2}} and omitted.
\qed

%

It is easy to compute that $c_1=\frac{1}{z}$ and $c_2=\frac{4z+9}{z^2}$ in \eqref{eq:c_v}.
Then we know when $p$ is an odd prime and $\gcd(p,z)=1$,
\begin{equation}\label{eq:1(2k+1)^2}
\sum_{k=0}^{p-1} (2k+1)^2 D_k(z)\equiv \frac{1}{z}\left(\frac{-z}{p}\right) \pmod p
\end{equation}
and
\[
\sum_{k=0}^{p-1}(2k+1)^4 D_k(z)\equiv \frac{4z+9}{z^2} \left(\frac{-z}{p}\right) \pmod {p}.
\]
Congruence \eqref{eq:1(2k+1)^2} was first discovered by Sun in \cite{Sun2014}.
It is also routine to compute that $\tilde{c}_1=-\frac{1}{z+1}$ and $\tilde{c}_2=-\frac{4z-5}{(z+1)^2}$ in \eqref{eq:c_v}.
Namely, when $p$ is an odd prime and $\gcd(p,z+1)=1$, we have
\begin{equation*}\label{eq:(2k+1)^2}
\sum_{k=0}^{p-1} (-1)^k (2k+1)^2 D_k(z)\equiv -\frac{1}{z+1}\left(\frac{z+1}{p}\right) \pmod p
\end{equation*}
and
\[
\sum_{k=0}^{p-1} (-1)^k (2k+1)^4 D_k(z)\equiv -\frac{4z-5}{(z+1)^2} \left(\frac{z+1}{p}\right) \pmod {p}.
\]
\section{Supercongruences for Delannoy numbers}
This section focuses on supercongruences for $\sum_{k=0}^{n-1}\varepsilon^k(2k+1)^{2v+1}D_k$ when $n$ is power of $2$.
We will provide an algorithmic approach to constructing new supercongruences and  confirm Guo and Zeng's conjecture \eqref{eq:GZconj}.

In 2018, Sun \cite[Lemma 2.1]{Sun2018} proved that
\begin{equation}\label{eq:DnSn}
D_{n+1}(z)-D_{n-1}(z)=2z(2n+1)S_n(z)
\end{equation}
for any positive integer $n$.
Here $S_n(z)$ are the large Schr\"oder polynomials defined by Sun \cite{Sun2012, Sun2018} as
\begin{equation*}
S_n(z)=\sum_{k=0}^{n}\binom{n}{k}\binom{n+k}{k}\frac{1}{k+1}z^k.
\end{equation*}
One can see $S_n(1)$ reduces to the $n$-th large Schr\"oder number $S_n$.

In 2021, Lengyel \cite{Lengyel2021} proved that, when $a\geq 2$,
\begin{equation}\label{eq:D_{2^a}}
 D_{2^a}\equiv 1 \pmod {4^{a+1} }
\end{equation}
and
\begin{equation}\label{eq:S_{2^a}}
 S_{2^a}\equiv 2-2^{a+1} \pmod {2^{2a+1} }.
\end{equation}
The following theorem considers supercongruences of $D_{2^a+1}$ and $D_{2^a-1}$ modulo $4^{a+1}$.
\begin{theo}
Let $a\geq 2$ be a positive integer. Then we have
\begin{equation}\label{eq:D_{2^{a}+1}}
D_{2^a+1}\equiv 3+2^{a+2} \pmod {4^{a+1}}
\end{equation}
and
\begin{equation}\label{eq:D_{2^{a}-1}}
D_{2^a-1}\equiv -1 \pmod {4^{a+1}}.
\end{equation}
\end{theo}
\pf
Note that
\[
D_n=\sum_{k=0}^{n}\binom{n}{k}^22^k,
\]
which can be proved by checking the recurrence relation and the initial values.
Let $T(n,k)=\binom{n}{k}^22^k$. When $k\geq 1$,
\[
T(2^a+1,k)=\binom{2^a+1}{k}^22^k=
\frac{(2^a+1)^2}{k^2}\binom{2^a}{k-1}^22^k.
\]
Then one can see
\[
v_2(T(2^a+1,k))=2a+k-2(v_2(k-1)+v_2(k)),
\]
where $v_2(m)$ is the largest $t\in\bN$ such that $2^t \mid m$.

When $k\geq 6$, we have $k-2(v_2(k-1)+v_2(k))\geq 2$ and thus
\begin{equation*}\label{eq:D_{2^a+1}}
D_{2^a+1}=\sum_{k=0}^{2^a+1}T(2^a+1,k)\equiv\sum_{k=0}^{5}T(2^a+1,k) \pmod {4^{a+1}}.
\end{equation*}
It is straightforward to calculate that when $a\geq 2$,
\begin{align*}
  T(2^a+1,0) & = 1, \\
  T(2^a+1,1) & =2(4^a+2^{a+1}+1), \\
  T(2^a+1,2) & =4^a(2^a+1)^2\equiv 4^a\pmod {4^{a+1}}, \\
  T(2^a+1,3) & \equiv2\cdot4^a(4^a-1)^2\equiv 2\cdot4^a\pmod {4^{a+1}},\\
  T(2^a+1,4) & \equiv4^a(4^a-1)^2(2^{a-1}-1)^2\equiv 4^a\pmod {4^{a+1}},\\
  T(2^a+1,5) & \equiv2\cdot4^a(4^a-1)^2(2^{a-1}-1)^2(2^a-3)^2\equiv 2\cdot4^a\pmod {4^{a+1}}.
\end{align*}
Summing all above congruences together leads to \eqref{eq:D_{2^{a}+1}} after simplification.

Let $z=1$ and $n=2^a$ in \eqref{eq:DnSn}, we arrive at
\begin{equation}\label{eq:D2aS2a}
D_{2^a-1}=D_{2^a+1}-2(2^{a+1}+1)S_{2^a}.
\end{equation}
Substituting $\eqref{eq:S_{2^a}}$ and \eqref{eq:D_{2^{a}+1}} into \eqref{eq:D2aS2a} leads to \begin{equation*}
D_{2^a-1}\equiv 3+2^{a+2}-(2^{a+1}+1)(4-2^{a+2})\equiv -1 \pmod {4^{a+1}}.
\end{equation*}
\qed

Now we are ready to prove Theorem \ref{th:1n^3}.

\noindent\emph{Proof of \eqref{eq:1main2} in Theorem \ref{th:1n^3}.}
When $a=1$, namely, $n=2$, we know that \eqref{eq:1main2} holds as
\[
\sum_{k=0}^{n-1}(2k+1)^{2v+1}D_k=1+9^{v+1}\equiv2\pmod 8
\]
by the fact that $D_0=1$ and $D_1=3$.
In the following, we assume $a\geq 2$.
By taking $\varepsilon=1$, $z=1$ and $s=2v$ and recursively using \eqref{eq:Explicit form of L*(x_s(k))2}, one can see
\begin{equation*}\label{eq:y(k-1)1}
(2k+1)^{2v+1}=L^{\ast}(y_v(k))
\end{equation*}
with $y_v(k)$ defined by \eqref{eq:yvk0}.
Note that $y_v(k)$ is a linear combination of even powers of $(2k+3)$ and $\deg_k(y_v(k))=2v$.
Then one can see
\begin{equation*}
y_v(k-1)=y_v(-k-2)
\end{equation*}
and thus assume
\begin{equation}\label{eq:y(k-1)2}
y_v(k-1)=\sum_{j=1}^{2v}s^{(v)}_jk^j+\frac{s^{(v)}_0}{2}\quad \text{and}  \quad
y_v(k-2)=\sum_{j=1}^{2v}s^{(v)}_j(-k)^j+\frac{s^{(v)}_0}{2}
\end{equation}
with $s^{(v)}_j\in\bZ$ for any $0\leq j\leq 2v$ since $e_j^{(2v)}$ are even by the definition \eqref{eq:ejs}.
Combining \eqref{eq:yvk0} and \eqref{eq:y(k-1)2} together shows that $s^{(v)}_0$ is an odd integer given by
\begin{equation}\label{eq:s0}
s^{(v)}_0=2\sum_{j=1}^{v}e_j^{(2v)} y_{v-j}(-1)-1.
\end{equation}

Taking $z=1$, $\varepsilon=1$ and $x(k)=y_v(k)$ in \eqref{eq: D1} gives
\begin{align*}
  \sum_{k=0}^{n-1}(2k+1)^{2v+1}D_k=& \sum_{k=0}^{n-1} L^{\ast}(y_v(k))D_k \\
  =&n\left(y_v(n-1)D_{n-1}-y_v(n-2)D_n\right) \\
  \equiv&n\left((s^{(v)}_1n+\frac{s^{(v)}_0}{2})D_{n-1}
            -(-s^{(v)}_1n+\frac{s^{(v)}_0}{2})D_n\right)\\
  \equiv & n\left(s^{(v)}_1n(D_n+D_{n-1})-\frac{s^{(v)}_0}{2}(D_n-D_{n-1})\right)\\
  \equiv &\rho_vn\pmod{n^3}
\end{align*}
with $\rho_v=-s^{(v)}_0$ as $D_{2^a}-D_{2^a-1}\equiv 2 \pmod {4^{a+1}}$ and $D_{2^a}+D_{2^a-1}\equiv 0 \pmod {4^{a+1}}$
by \eqref{eq:D_{2^a}} and \eqref{eq:D_{2^{a}-1}}.
This completes the proof.
\qed

Note that $\rho_v$ can be determined automatically.
For example, by taking $v=0,1,2$ in \eqref{eq:pv} respectively, we obtain
\begin{align*}
  &\sum_{k=0}^{n-1}(2k+1)D_k \equiv n\pmod{n^3}, \\
  &\sum_{k=0}^{n-1}(2k+1)^3D_k \equiv 5n\pmod{n^3}, \\
  &\sum_{k=0}^{n-1}(2k+1)^5D_k \equiv 105n\pmod{n^3}.
\end{align*}

\noindent\emph{Proof of \eqref{eq:-1main2} in Theorem \ref{th:1n^3}.}
When $n=2$, we know \eqref{eq:-1main2} holds as
\[
\sum_{k=0}^{n-1}(-1)^k(2k+1)^{2v+1}D_k=1-9^{v+1}\equiv 0 \pmod 8.
\]
In the following, we assume $n\geq 4$.
Let $\varepsilon=-1$, $z=1$ and $s=2v$ in \eqref{eq:Explicit form of L*(x_s(k))2}.
Then by recursively using \eqref{eq:Explicit form of L*(x_s(k))2}, we obtain that
\[
(2k+1)^{2v+1}=L^{\ast}(\tilde{y}_v(k))
\]
with $\tilde{y}_v(k)$ given by \eqref{eq:yvk}.
Since $\tilde{y}_v(k)$ is a linear combination of even power of $(2k+3)$,
we can assume
\begin{equation}\label{eq:y[k-1]}
\tilde{y}_v(k-1)=\sum_{j=1}^{2v}\tilde{s}^{(v)}_jk^j+\frac{\tilde{s}_0}{4}\quad \text{and}  \quad
\tilde{y}_v(k-2)=\sum_{j=1}^{2v}\tilde{s}^{(v)}_j(-k)^j+\frac{\tilde{s}_0}{4}
\end{equation}
with $\tilde{s}^{(v)}_j\in\bZ$ for any $0\leq j\leq 2v$.
Combining \eqref{eq:yvk} and \eqref{eq:y[k-1]} together gives
\[
\tilde{s}^{(v)}_1=v-\sum_{j=1}^{v}\frac{e^{(2v)}_j}{2}
\tilde{y}_{v-j}'(-1)\in\bZ.
\]
Taking $z=1$, $\varepsilon=-1$ and $x(k)=y_v(k)$ in \eqref{eq: D1} gives
\begin{align*}
  \sum_{k=0}^{n-1}(-1)^k(2k+1)^{2v+1}D_k=& \sum_{k=0}^{n-1} L^{\ast}(y_v(k))(-1)^kD_k \\
  =&n\left(y_v(n-1)(-1)^{n-1}D_{n-1}-y_v(n-2)(-1)^{n}D_n\right) \\
  \equiv&n\left(-(\tilde{s}^{(v)}_1n+\frac{\tilde{s}^{(v)}_0}{4})D_{n-1}
            -(-\tilde{s}^{(v)}_1n+\frac{\tilde{s}^{(v)}_0}{4})D_n\right)\\
  \equiv & n\left(\tilde{s}^{(v)}_1n(D_n-D_{n-1})-\frac{\tilde{s}^{(v)}_0}{4}
  (D_n+D_{n-1})\right)\\
  \equiv &\tilde{\rho}_vn^2\pmod{n^3}
\end{align*}
with $\tilde{\rho}_v=2\tilde{s}^{(v)}_1$ with the help of \eqref{eq:D_{2^a}} and \eqref{eq:D_{2^{a}-1}}.
This completes the proof.
\qed

The constant $\tilde{\rho}_v$ can also be calculated mechanically.
By taking $v=0,1,2$ in \eqref{eq:pv2} respectively, we arrive at
\begin{align}\label{eq:v=1}
  &\sum_{k=0}^{n-1}(-1)^k(2k+1)D_k \equiv 0\pmod{n^3},\nonumber \\
  &\sum_{k=0}^{n-1}(-1)^k(2k+1)^3D_k \equiv 2n^2\pmod{n^3}, \\
  &\sum_{k=0}^{n-1}(-1)^k(2k+1)^5D_k \equiv -12n^2\pmod{n^3}\nonumber.
\end{align}
Note that \eqref{eq:v=1} confirms Guo and Zeng's conjecture \eqref{eq:GZconj}.

\noindent \textbf{Acknowledgments.}
This work was supported by the National Natural Science Foundation of China (No. 12101449, 12271511 and 12271403).
All authors contribute equally and the names of the authors are listed in alphabetical order according to their family names.

\end{document}